\documentclass[12pt]{amsart}
\usepackage{amssymb,amsfonts,amsthm,latexsym,bm}
\usepackage{hyperref}
\usepackage[centertags]{amsmath}
\usepackage{mathrsfs}
\usepackage{t1enc}
\usepackage{graphicx}
\usepackage{enumerate}

\advance\headheight by 3pt

\newtheorem{theorem}{Theorem}[section]
\newtheorem{corollary}[theorem]{Corollary}
\newtheorem{lemma}[theorem]{Lemma}

\newtheorem*{theorem*}{Theorem}

\numberwithin{equation}{section}

\def\pitem{\advance\leftskip3mm\advance\linewidth-3mm}
\def\mitem{\advance\leftskip-3mm\advance\linewidth3mm}

\begingroup\makeatletter\ifx\SetFigFont\undefined
\gdef\SetFigFont#1#2#3#4#5{
  \reset@font\fontsize{#1}{#2pt}
  \fontfamily{#3}\fontseries{#4}\fontshape{#5}
  \selectfont}
\fi\endgroup

\theoremstyle{definition}

\renewcommand{\subjclass}[1]{\thanks{\emph{2020 Mathematics Subject Classification:}~#1}}
\renewcommand{\keywords}[1]{\thanks{\emph{Keywords and Phrases:}~#1}}
\renewcommand{\date}{\thanks{\today}}

\newcommand{\av}{\mathbf{a}}
\newcommand{\bv}{\mathbf{b}}
\newcommand{\cv}{\mathbf{c}}
\newcommand{\dv}{\mathbf{d}}
\newcommand{\ev}{\mathbf{e}}

\newcommand{\uv}{\mathbf{u}}

\newcommand{\xv}{\mathbf{x}}
\newcommand{\yv}{\mathbf{y}}
\newcommand{\zv}{\mathbf{z}}

\renewcommand{\AA}{\mathcal{A}}
\newcommand{\BB}{\mathcal{B}}
\newcommand{\CC}{\mathcal{C}}

\newcommand{\EE}{\mathcal{E}}

\newcommand{\HH}{\mathcal{H}}

\newcommand{\KK}{\mathcal{K}}
\newcommand{\LL}{\mathcal{L}}

\newcommand{\NN}{\mathcal{N}}

\newcommand{\TT}{\mathcal{T}}
\newcommand{\UU}{\mathcal{U}}
\newcommand{\VV}{\mathcal{V}}

\newcommand{\fp}{\mathfrak{p}}

\newcommand{\Cc}{\mathbb{C}}
\newcommand{\Rr}{\mathbb{R}}
\newcommand{\Qq}{\mathbb{Q}}

\newcommand{\Zz}{\mathbb{Z}}

\newcommand{\ordfp}{{\rm ord}_{\mathfrak{p}}}

\def\house#1{\setbox1=\hbox{$\,#1\,$}%
\dimen1=\ht1 \advance\dimen1 by 2pt \dimen2=\dp1 \advance\dimen2
by 2pt
\setbox1=\hbox{\vrule height\dimen1 depth\dimen2\box1\vrule}%
\setbox1=\vbox{\hrule\box1}%
\advance\dimen1 by .4pt \ht1=\dimen1 \advance\dimen2 by .4pt
\dp1=\dimen2 \box1\relax}

\newcommand{\kdots}{,\ldots ,}

\newcommand{\half}{\mbox{$\textstyle{\frac{1}{2}}$}}

\newcommand{\medfrac}[2]{\mbox{\large{$\textstyle{\frac{#1}{#2}}$}}}

\newcommand{\rank}{{\rm rank}\,}

\newcommand{\Z}{\mathbb{Z}}

\title[Sums of elements from a multiplicative group]{Asymptotic formulas for sums of elements from a multiplicative group}

\author[J.-H. Evertse]{Jan-Hendrik Evertse}

\author[K. Gy\H{o}ry]{K\'alm\'an Gy\H{o}ry}

\author[L. Hajdu]{Lajos Hajdu}

\author[F. Luca]{Florian Luca}

\author[L. Remete]{L\'aszl\'o Remete}

\thanks{Research supported in part by the HUN-REN Hungarian Research Network and by the NKFIH grants 128088 and 130909.}



\address{Jan-Hendrik Evertse
\hfill\break\indent Leiden University, Mathematical Institute
\hfill\break\indent 2300 RA Leiden, P.O. Box 9512.
\hfill\break\indent The Netherlands}

\email{evertse@math.leidenuniv.nl}

\address{K\'alm\'an Gy\H{o}ry
\hfill\break\indent University of Debrecen, Institute of Mathematics
\hfill\break\indent H-4002 Debrecen, P.O. Box 400.
\hfill\break\indent Hungary}

\email{gyory@science.unideb.hu}

\address{Lajos Hajdu
\hfill\break\indent University of Debrecen, Institute of Mathematics
\hfill\break\indent H-4002 Debrecen, P.O. Box 400.
\hfill\break\indent and HUN-REN DE Equations, Functions, Curves and their Applications Research Group
\hfill\break\indent Hungary}

\email{hajdul@science.unideb.hu}

\address{Florian Luca
\hfill\break\indent Stellenbosch University
\hfill\break\indent Mathematics/Industrial Psychology Building 
\hfill\break\indent
Merriman Street, 7600 Stellenbosch, South Africa}

\email{fluca@sun.ac.za}

\address{L\'aszl\'o Remete
\hfill\break\indent University of Debrecen, Institute of Mathematics
\hfill\break\indent H-4002 Debrecen, P.O. Box 400.
\hfill\break\indent and HUN-REN DE Equations, Functions, Curves and their Applications Research Group
\hfill\break\indent Hungary}

\email{laszlo.remete@science.unideb.hu}

\begin{document}

\begin{abstract}
Let $K$ be a number field, $k\geq 2$ an integer, $(K^*)^k$ the $k$-fold direct product of $K^*$ with coordinatewise multiplication, and $\Gamma$ a finitely generated subgroup of rank $r$ of $(K^*)^k$. Further, let $H(\alpha )$ denote the absolute exponential height of an algebraic number $\alpha$. Fix non-zero elements $a_1\kdots a_k\in K$. 
We give asymptotic formulas for the number of $\mathbf{x}=(x_1\kdots x_k)\in\Gamma$ with $H(a_1x_1+\cdots +a_kx_k)\leq X$ as $X\to\infty$ such that no non-empty subsum of $a_1x_1+\cdots +a_kx_k$ vanishes. By the same method of proof, we obtain an asymptotic formula as $X\to\infty$ for the number of non-negative integers $n$ with $H(u_n)\leq X$, where $\{ u_n\}$ is a linear recurrence sequence.
\end{abstract}

\maketitle

\section{Introduction}\label{sec1} 
Let $K$ be a number field, $k\geq 2$ an integer, and $S$ a finite set of places of $K$, containing all infinite places. Suppose $S$ has cardinality $s+1$. Denote by $U_S$ the group of $S$-units of $K$; then $U_S$ has rank $s$. Fix non-zero $a_1\kdots a_k\in K$.
Denote by $H(\alpha )$ the absolute exponential Weil height of an algebraic number $\alpha$.

This paper builds further on work of G. Everest. 
By a slight modification of the arguments in his papers \cite{everest1,everest}, one can deduce an asymptotic formula for the number of tuples $(x_1\kdots x_k)$ with $x_1\kdots x_k\in U_S$ that satisfy
\[
H(a_1x_1+\cdots +a_kx_k)\leq X
\]
and are non-degenerate in the sense that no non-trivial subsum of $\sum_{i=1}^k a_ix_i$ vanishes.
In fact, Everest's arguments yield for this number an asymptotic formula
\begin{equation}\label{eq:1.1}
c(U_S,k)(\log X)^{ks}+O((\log X)^{ks-1})\ \ \text{as } X\to\infty,
\end{equation}
where $c(U_S,k)>0$. Everest's main tools are a result from \cite{evertse} that compares $H(a_1x_1+\cdots +a_kx_k)$ with the height of the vector $(x_1\kdots x_k)$ (which in turn was deduced from Schmidt's and Schlickewei's $p$-adic Subspace theorem \cite{schmidt, schlickewei}), and  logarithmic forms estimates. 

Let $(K^*)^k$ denote the $k$-fold direct product of $K^*$ with coordinatewise multiplication and inversion.
We consider more generally sums 
\[
a_1x_1+\cdots +a_kx_k\ \ \text{with $(x_1\kdots x_k)\in\Gamma$,}
\]
where $\Gamma$ is a finitely generated subgroup of rank $r$ of $(K^*)^k$.
We assume that for each distinct $i,j\in\{ 1\kdots k\}$, there is 
$(x_1\kdots x_k)\in\Gamma$ such that $x_i/x_j$ is not a root of unity.
The purpose of this paper is to deduce an asymptotic formula for the number of non-degenerate $(x_1\kdots x_k)\in\Gamma$ with $H(a_1x_1+\cdots +a_kx_k)\leq X$. 
Our first result (see Theorem \ref{thm:2.1} below) is, that this number equals 
\begin{equation}\label{eq:1.2}
c(\Gamma )(\log X)^r\cdot\left(1+O\Big(\frac{(\log\log X)^2}{\log X}\Big)^{1/(r+4)}\right)\ \ \text{as } X\to\infty, 
\end{equation}
with $c(\Gamma )>0$ depending only on $\Gamma$. In this general set-up, we could not make Everest's method work, so we used instead  
a different method, based on a suitable quantitative version of the Parametric Subspace Theorem \cite{evfer}. 

By a similar argument, we obtain an asymptotic formula for non-degenerate linear recurrence sequences
$u_n=\sum_{i=1}^k a_i(n)\alpha_i^n$, i.e., the $a_i$ are non-zero polynomials in $K[X]$
and the $\alpha_i$ non-zero elements of $K$, such that for any two distinct $i,j$,
$\alpha_i/\alpha_j$ is not a root of unity. In fact, we show that 
the number of non-negative integers $n$ with $H(u_n)\leq X$ equals
\begin{equation}\label{eq:1.2a}
c\log X\left(1+O\Big(\frac{(\log\log X)^2}{\log X}\Big)^{1/4}\right)\ \ \text{as } X\to\infty , 
\end{equation}
where $c>0$, see Theorem \ref{thm:2.2} below.

We can improve on \eqref{eq:1.2} 
in the special case 
\[
\Gamma =\Gamma_1^k =\{(x_1\kdots x_k):\, x_1\kdots x_k\in\Gamma_1\},
\]
where $\Gamma_1$ is a finitely generated subgroup of $K^*$. Indeed, in this situation, Everest's method is applicable, and we obtain instead of \eqref{eq:1.2},
\begin{equation}\label{eq:1.3}
c(\Gamma )(\log X)^r +O((\log X)^{r-1})\ \ \text{as } X\to\infty ,
\end{equation}
see Theorem \ref{thm:2.3} below.
In fact, our method of proof is a simplification of Everest's, which avoids the logarithmic forms estimates. In the case $\Gamma_1=U_S$ we obtain \eqref{eq:1.1}.

Finally, we deduce an asymptotic formula where we do not count tuples
$(x_1\kdots x_k)$, but instead algebraic numbers $\alpha$ with $H(\alpha )\leq X$ that are representable in the form  $a_1x_1+\cdots +a_kx_k$ with $x_1\kdots x_k\in\Gamma_1$, see Theorem \ref{thm:2.4} below. This extends work of \'Ad\'am, Hajdu, and Luca \cite{ahl}, and of Frei, Tichy, and Ziegler \cite{ftz}.

In all our theorems, the constants implied by the $O$-symbols in the asymptotic formulas cannot be effectively determined by our method of proof,
except for the one in Theorem \ref{thm:2.2} on linear recurrence sequences, where the implied constant is effectively computable.

In the next section, we introduce the necessary notation and state the above results in a precise form. In Section \ref{sec3} we have collected our tools, and in the subsequent sections, we prove our theorems. 

\section{Notation and results}\label{sec2}

We denote by $|\AA |$ the cardinality of a set $\AA $.
Throughout this paper, $K$ will be an algebraic number field of degree $d$ and $k\geq 2$ an integer. Let $K^*$ denote the group $K\setminus\{ 0\}$ with multiplication, and $(K^*)^k$ its $k$-fold direct product. Denote by $O_K$ the ring of integers of $K$.
The real infinite places of $K$ are its real embeddings $\sigma_i: K\hookrightarrow \Rr$
($i=1\kdots r_1$), the complex infinite places of $K$ are its pairs of conjugate complex embeddings 
\[
\{ \sigma_i ,\overline{\sigma_i} :\, K\hookrightarrow\Cc\}\ \ (i=r_1+1\kdots r_1+r_2),
\]
and the finite places of $K$ are the prime ideals of $O_K$. The set of places $M_K$ of $K$ consists of the (real and complex) infinite places and the finite places of $K$. We define normalized absolute values $\|\cdot \|_v $ $(v\in M_K)$ on $K$ by
\begin{align*}
&\|\cdot \|_v:=|\sigma (\cdot )|^{1/d}\ &\text{if $v=\sigma$ is real infinite,}
\\
&\|\cdot \|_v:=|\sigma (\cdot )|^{2/d}\ &\text{if $v=\{ \sigma ,\overline{\sigma}\}$ is complex infinite,}
\\
&\|\cdot \|_v:=N\fp^{-\ordfp (\cdot )/d}\ &\text{if $v=\fp$ is finite,}
\end{align*}
where $N\fp =|O_K/\fp |$ is the norm of $\fp$, and $\ordfp (\alpha )$ is the exponent of $\fp$
in the unique prime ideal factorization of $\alpha\in K$, with $\ordfp (0)=\infty$.
These absolute values satisfy the product formula
\[
\prod_{v\in M_K} \|\alpha \|_v=1\ \ \text{for } \alpha\in K^*.
\]

The absolute, exponential  height and logarithmic height of $\alpha\in K$ are given by
$$
H(\alpha) :=\prod\limits_{v\in M_K} \max(1,\|\alpha\|_v),\ \ h(\alpha ):=\log H(\alpha ).
$$
These heights depend only on $\alpha$, i.e., are independent of the choice of a number field $K$ containing $\alpha$.
For later use, note that (as is well-known), we have
$$
H(\alpha_1\cdots\alpha_n)\leq H(\alpha_1)\cdots H(\alpha_n)
$$
and
$$
H(\alpha_1+\cdots+\alpha_n)\leq nH(\alpha_1)\cdots H(\alpha_n)
$$
for any $\alpha_1,\dots,\alpha_n\in K$. Further, for any non-zero $\alpha\in K$ and $m\in\Z$ we have
$$
H(\alpha^m)=H(\alpha)^{|m|}.
$$
More generally, for a vector $\mathbf{x}=(x_1\kdots x_k)\in K^k$ we define
\[
\|\mathbf{x}\|_v:=\max (\|x_1\|_v\kdots \|x_k\|_v)\ \ \text{for } v\in M_K
\] 
and subsequently its exponential and logarithmic height
\[
H(\mathbf{x})=H(x_1\kdots x_k):=\prod_{v\in M_K}\max (1,\|\mathbf{x}\|_v),\ \ h(\xv ):=\log H(\xv ).
\]
The product of $\mathbf{x}=(x_1\kdots x_k),\, \mathbf{y}=(y_1\kdots y_k)\in K^k$ is defined coordinatewise, i.e., $\mathbf{x}\cdot\mathbf{y}=(x_1y_1\kdots x_ky_k)$. 
Further, if $\mathbf{x}\in (K^*)^k$, i.e., $x_1\cdots x_k\not= 0$ we define $\mathbf{x}^{-1}:=(x_1^{-1}\kdots x_k^{-1})$.  
One easily shows that
\begin{align}\label{eq:2.1}
&H(\mathbf{x}\cdot \mathbf{y})\leq H(\mathbf{x})\cdot H(\mathbf{y})\ \ \text{for }
\mathbf{x},\mathbf{y}\in K^k,
\\
\label{eq:2.2}
&H(\mathbf{x}^m)=H(\mathbf{x})^m\ \ \text{for } \mathbf{x}\in K^k,\, m\in\Zz_{\geq 0}.
\end{align}
Further,
\begin{equation}\label{eq:2.3}
H(\mathbf{x}^{-1})\leq H(\mathbf{x})^{k}\ \ \text{for } \mathbf{x}\in (K^*)^k.
\end{equation}
Indeed, if $\mathbf{x}=(x_1\kdots x_k)$ one may write for $v\in M_K$,
\begin{align*}
&\max (1,\|\mathbf{x}^{-1}\|_v)=\max (1,\| x_1\|_v^{-1}\kdots \| x_k\|_v^{-1}) 
\\
&\quad =\|x_1\cdots x_k\|_v^{-1}\max \Big(\|x_1\cdots x_k\|_v,\|\prod_{j\not= 1}x_j\|_v\kdots \|\prod_{j\not= k}x_j\|_v\Big)
\\
&\quad \leq \|x_1\cdots x_k\|_v^{-1}\max (1,\|\xv\|_v)^k
\end{align*}
and by taking the product over all $v$ and applying the product formula, one obtains \eqref{eq:2.3}.

Let $S$ be a finite subset of $M_K$ with $|S|\geq 2$, containing all infinite places.
The multiplicative group of $S$-units in $K$ is defined by
\begin{align*}
U_S:=\{ u\in K:\ \| u\|_v= 1\ \text{for } v\in M_K\setminus S \}.
\end{align*}
It is well-known that $U_S$ is finitely generated and that its rank is $s:=|S|-1$.
Further, the torsion group of $U_S$ is the group $W_K$ of roots of unity in $K$. 
Denote by $\omega_K$ the order of $W_K$.

Let $S=\{ v_0,\kdots v_s\}$, and let $\{ \gamma_1\kdots \gamma_s\}$ be a fundamental set of $S$-units, or equivalently, a basis of $U_S/W_K$, i.e., every $u\in U_S$ can be expressed uniquely as $\zeta \gamma_1^{z_1}\cdots\gamma_r^{z_r}$ with $\zeta\in W_K$ and $z_1\kdots z_r\in\Zz$. Define the quantity
\begin{equation}\label{eq:2.3b}
R_S :=\left|\det\left(\begin{array}{ccc}
\log \|\gamma_1\|_{v_1}&\cdots &\log \|\gamma_s\|_{v_1}
\\
\vdots&\cdots&\vdots
\\
\log \|\gamma_1\|_{v_s}&\cdots &\log \|\gamma_s\|_{v_s}
\end{array}\right)\right|.
\end{equation}
This is independent of the choice of $v_1\kdots v_s$, 
$\gamma_1\kdots \gamma_s$.
For properties of the introduced notions, see e.g. Chapters 2.1 and 2.2 of Evertse and Gy\H{o}ry \cite{egybook}. The $S$-regulator as defined in \cite{egybook} is $d^sR_S$.

Now let $\Gamma$ be a finitely generated subgroup of rank $r$ of $(K^*)^k$.
Write elements of $\Gamma$ as $\xv =(x_1\kdots x_k)$.
Choose a basis $\uv_1\kdots \uv_r$ of 
$\Gamma/\Gamma_{{\rm tors}}$, by which we mean that $\uv_1\kdots\uv_r\in\Gamma$ and every $\mathbf{x}\in\Gamma$ can be expressed uniquely as 
\begin{equation}\label{eq:2.3a}
\mathbf{\zeta}\cdot \mathbf{u}_1^{z_1}\cdots\mathbf{u}_r^{z_r}\ \ \text{with } \mathbf{\zeta}\in\Gamma_{{\rm tors}},\ \ z_1\kdots z_r\in\Zz .
\end{equation}
Write 
$\mathbf{u}_i=(u_{i1}\kdots u_{ik})$ for $i=1\kdots k$.
Let $S$ be the smallest set of places of $K$ such that $\| u_{ij}\|_v=1$ for $v\in M_K\setminus S$, $i=1\kdots r$, $j=1\kdots k$. Then $S$ is finite, and we have
\begin{equation}\label{eq:3.6}
S=\{ v\in M_K :\, \text{there is $\xv\in\Gamma$ with $\| x_i\|_v\not= 1$ for at least one $i$}\}.
\end{equation}
Now for the logarithmic height of $\mathbf{x}\in\Gamma$, we have
\begin{align}\label{eq:2.4}
&h(\mathbf{x})=
\sum_{v\in S}\max \big(0,\ell_{v1}(\textbf{z})\kdots \ell_{vk}(\textbf{z})\big),
\\[-0.15cm]
\notag
&\quad\text{where } \ell_{vj}(\textbf{z}):=\sum_{i=1}^r z_i\log \|u_{ij}\|_v\ \ (v\in S,\, j=1\kdots k).
\end{align}
Let $\mu$ denote the Lebesgue measure on $\Rr$ with $\mu ([0,1])=1$, and $\mu^{r}$ the product measure on $\Rr^r$. Define
\begin{align}\label{eq:2.5}
&\CC (\Gamma):=\big\{ \mathbf{\xi}\in\Rr^r:\, \sum_{v\in S}\max \big(0,\ell_{v1}(\mathbf{\xi})\kdots \ell_{vk}(\mathbf{\xi})\big)\leq 1\big\},
\\
\label{eq:2.6}
&c(\Gamma ):= |\Gamma_{{\rm tors}}|\cdot \mu^r(\CC(\Gamma )).
\end{align}
The product formula implies $\sum_{v\in S}\ell_{vj}=0$ for $j=1\kdots r$, and from this one easily deduces that $\CC (\Gamma )$ is a bounded set. Hence, $c(\Gamma )$ is finite. The set $\CC (\Gamma )$ depends on the choice of $\mathbf{u}_1\kdots \mathbf{u}_r$, but the quantity $c(\Gamma )$ is independent of this choice.

Henceforth we write $\mathbf{x}=(x_1\kdots x_k)$,
$\mathbf{a}=(a_1\kdots a_k)$, etc.
For $\mathbf{a}\in (K^*)^k$
and $X>1$, define the set
\begin{align*}
\VV_{\Gamma ,\mathbf{a}}(X):=\Big\{ \xv\in\Gamma :\, &H(a_1x_1+\cdots +a_kx_k)\leq X,\ 
\\
&\sum_{i\in I}a_ix_i\not= 0\ \text{for each non-empty } I\subseteq\{ 1\kdots k\}\Big\}.
\end{align*}

Our first result is as follows.

\begin{theorem}\label{thm:2.1}
Let $K$ be a number field, $k\geq 2$ an integer, and
$\Gamma$ a finitely generated subgroup of $(K^*)^k$ of rank $r\geq 1$.
Assume that
\begin{align}\label{eq:2.7}
&\text{for each $i\not= j\in\{ 1\kdots k\}$ there is $\mathbf{x}\in\Gamma$} 
\\[-0.1cm]
\notag
&\text{such that $x_i/x_j$ is not a root of unity.}
\end{align}
Let $\mathbf{a}=(a_1\kdots a_k)\in (K^*)^k$. Then
\begin{align*}
|\VV_{\Gamma ,\mathbf{a}}(X)| =\, & 
c(\Gamma )(\log X)^r\cdot\left(1+O\Big(\Big(\frac{(\log\log X)^2}{\log X}\Big)^{1/(r+4)}\Big)\right)\
\\[-0.1cm]
&\text{as } X\to\infty, 
\end{align*}
where $c(\Gamma) $ is defined by \eqref{eq:2.6}. Here, the implied constant depends only on $\Gamma$ and $\mathbf{a}$.  
\end{theorem}

By specializing to the case $\Gamma =\{ (\alpha_1^n\kdots \alpha_k^n):\, n\in\Zz\}$, i.e., $r=1$, 
we can deduce a result for linear recurrence sequences $u_n=\sum_{i=1}^k a_i\alpha_i^n$.
By modifying the proof of Theorem \ref{thm:2.1} we managed to improve and generalize this.

\begin{theorem}\label{thm:2.2}
Let $K$ be a number field, $k\geq 2$ an integer,  
$a_1\kdots a_k\in K[X]$ non-zero polynomials, and $\alpha_1\kdots \alpha_k\in K^*$. Assume that 
\begin{equation}\label{eq:2.7a} 
\text{for each $i\not= j\in \{ 1\kdots k\}$, $\alpha_i/\alpha_j$ is not a root of unity.}
\end{equation} 
Put $u_n:=\sum_{i=1}^k a_i(n)\alpha_i^n$ for $n\in\Zz_{\geq 0}$.
Then
\begin{align*}
|\{ n\in\Zz_{\geq 0}:\, H(u_n)\leq X\}| =\, &\frac{\log X}{\log H}\cdot\left(1+O\Big(\Big(\frac{(\log\log X)^2}{\log X}\Big)^{1/4}\Big)\right)\ 
\\
&\text{as } X\to\infty,
\end{align*}
where
\[
H:=H(\alpha_1\kdots\alpha_k)=\prod_{v\in M_K}\max (1,\|\alpha_1\|_v\kdots \|\alpha_k\|_v),
\]
and the implied constant depends only on $a_1\kdots a_k,\alpha_1\kdots \alpha_k$ and can be determined effectively.
\end{theorem}

Further potential applications of our method of proof are to linear combinations of linear recurrence sequences $\{ b_1u_m+b_2v_n:\, m,n\in\Zz_{\geq 0}\}$ (see \cite{tvyz} for a related result) or perhaps even $\{ b_1u_{n_1,1}+\cdots +b_lu_{n_l,l}:\, n_1\kdots n_l\in\Zz_{\geq 0}\}$.
Most generally, one might consider exponential polynomials
\[
\Big\{ \sum_{i=1}^k f_i(n_1\kdots n_l)\alpha_{i1}^{n_1}\cdots \alpha_{il}^{n_l}:\, n_1\kdots n_l\in\Zz\Big\},
\]
where $f_i\in K[X_1\kdots X_l],\alpha_{i1}\kdots \alpha_{il}\in K^*$ for $i=1\kdots k$. 

We can sharpen Theorem \ref{thm:2.1} if we strengthen \eqref{eq:2.7}.
As before, we write $\xv =(x_1\kdots x_k)$, and put $x_0:=1$.

\begin{theorem}\label{thm:2.3a}
Let $K$ be a number field, $k\geq 2$ an integer, and
$\Gamma$ a finitely generated subgroup of $(K^*)^k$ of rank $r\geq 1$.
Let $S$ be given by \eqref{eq:3.6}. Assume that 
\begin{align}\label{eq:2.7b}
&\text{for each $v\in S$, $i\not= j\in\{ 0\kdots k\}$ there is $\xv\in\Gamma$ with} 
\\[-0.1cm]
\notag
&\text{$\|x_i\|_v\not=\|x_j\|_v$.}
\end{align}
Let $\mathbf{a}\in (K^*)^k$. Then
\[
|\VV_{\Gamma ,\mathbf{a}}(X)| = 
c(\Gamma )(\log X)^r +O((\log X)^{r-1})\ \ \text{as } X\to\infty . 
\]
Here, the implied constant depends only on $\Gamma$ and $\mathbf{a}$.  
\end{theorem}

An important special case is, when 
\[
\Gamma =\Gamma_1^k=\{ (x_1\kdots x_k):\, x_1\kdots x_k\in\Gamma_1\},
\]
where $\Gamma_1$ is a finitely generated subgroup of $K^*$. Note that in this case, $\rank\Gamma =k\cdot \rank\Gamma_1$.

\begin{theorem}\label{thm:2.3}
Let $K$ be a number field, $k\geq 2$ an integer, 
$\Gamma_1$ a subgroup of $K^*$ of rank $r_1\geq 1$, and $\mathbf{a}\in (K^*)^k$.
Then  
\[
|\VV_{\Gamma_1^k ,\mathbf{a}}(X)| = c(\Gamma_1^k )(\log X)^{kr_1}+O((\log X)^{kr_1-1})\ \ \text{as } X\to\infty .
\]
Here, the implied constant depends only on $\Gamma_1$, $k$, and $\mathbf{a}$.
\end{theorem}

In general, it seems to be a difficult problem to derive an explicit expression for $c(\Gamma )$, but this can be done in the case that $\Gamma =U_S^k$, where $S$ is a finite set of places of $K$, containing all infinite places. 
Write $S=\{ v_0\kdots v_s\}$ and pick a basis
$\{ \gamma_1\kdots \gamma_s\}$ of $U_S/W_K$.
Using the product formula in the form $\log\| u\|_{v_0}=-\sum_{j=1}^s\log\| u\|_{v_j}$ for $u\in U_S$, together with \eqref{eq:2.3b},
one can show that $\CC (U_S^k)$ is the image under a linear transformation of determinant $R_S^{-k}$ of the set
\[
\EE_{s,k}:=\Big\{ \mathbf{\xi}=(\xi_{v_i,j})_{i=1\kdots s, j=1\kdots k}\in\Rr^{ks}:\,
\sum_{i=0}^s\max (0,\xi_{v_i,1}\kdots \xi_{v_i,k})\leq 1\Big\},
\]
where $\xi_{v_0,j}=-\sum_{i=1}^s \xi_{v_i,j}$ for $j=1\kdots k$. Kerber, Tichy, and Weitzer \cite{ktw}
proved that this set (called by them `Everest polytope' since it was introduced by Everest \cite{everest1}) has measure $((k+1)s)!/(ks)!(s!)^{k+1}$. Noticing that $(U_S^k)_{{\rm tors}}=W_K^k$  and thus, has cardinality $\omega_K^k$, it follows that
\[
c(U_S^k)=\frac{\omega_K^k}{R_S^k}\cdot \frac{((k+1)s)!}{(ks)!(s!)^{k+1}}.
\]

In our next result, we do not count tuples $(x_1\kdots x_k)\in\Gamma_1^k$, but instead numbers $\alpha$ representable in the form $a_1x_1+\cdots +a_kx_k$. Further, we do not take a fixed tuple of coefficients $(a_1\kdots a_k)$, but instead a finite set of such tuples.

Denote by $S_k$ the permutation group of $\{ 1\kdots k\}$. Henceforth, we write
\begin{align*}
&\av=(a_1\kdots a_k),\ \ \bv =(b_1\kdots b_k),\ \ 
\av /\bv=(a_1/b_1\kdots a_k/b_k),
\\ 
&\sigma (\av )=(a_{\sigma (1)}\kdots a_{\sigma (k)})\ \ (\sigma\in S_k). 
\end{align*}
Let $\AA$ be a finite set of tuples in $(K^*)^k$ with  the following properties:
\begin{align}
\label{eq1.a}
&\text{for each $\av\in\AA,\,i,j\in\{ 1\kdots k\}$ we have $a_i=a_j$ or $a_i/a_j\not\in \Gamma_1$;}
\\
\label{eq1.c}
&\text{for each $\av ,\bv \in\AA$ we have $\av =\bv$ or $\av /\bv\not\in \Gamma_1^k$;}
\\
\label{eq1.b}
&\text{for each $\av\in\AA,\,\sigma\in S_k$ we have $\sigma (\av )\in\AA$.} 
\end{align}

From Theorem \ref{thm:2.3} we deduce the following result.
In the deduction we closely follow Frei, Tichy, and Ziegler \cite{ftz}.

\begin{theorem}\label{thm:2.4}
Let $K$ be a number field, $k$ an integer $\geq 1$, 
and $\Gamma_1$ a finitely generated subgroup of $K^*$ of rank $r_1>0$. Further, 
	let $\AA$ be a finite set of tuples in $(K^*)^k$ with \eqref{eq1.a}, \eqref{eq1.c}, \eqref{eq1.b}.
		Let $\TT_{\AA}(X)$ be the set of $\alpha\in K^*$ with the following properties:
	\begin{equation}\label{eq1.2}
	\left\{
	\begin{array}{l}
	H(\alpha )\leq X;
	\\[0.1cm]
	\text{there are $(a_1\kdots a_k)\in\AA$, $x_1\kdots x_k\in\Gamma_1$}
	\\	
	\text{such that $\alpha =a_1x_1+\cdots +a_kx_k$}.
	\end{array}\right.
	\end{equation}
		Then
	\[
	|\TT_{\AA}(X)|=
	\frac{|\AA|\cdot c(\Gamma_1^k )}{k!}(\log X)^{kr_1}+O((\log X)^{kr_1-1})\ \ \text{as } X\to\infty .
\]
Here, the implied constant depends only on $\Gamma_1$, $k$, and $\AA$.
\end{theorem}
\noindent
Note that in \eqref{eq1.2} we allow subsums of $a_1x_1+\cdots +a_kx_k$ to be $0$.

Two special cases of Theorem \ref{thm:2.4} are of interest:
\begin{itemize}
\item[(1)] Let $\AA =\BB^k$ ($k$ times cartesian product), where $\BB$ is a finite subset of $K^*$
such that for all $b,b'\in\BB$ with $b\not= b'$ we have $b/b'\not\in \Gamma_1$. Note that $\BB^k$ satisfies 
\eqref{eq1.a}, \eqref{eq1.c}, \eqref{eq1.b}. Hence,
\[
|\TT_{\BB^k}(X)|=\frac{|\BB |^k\cdot c(\Gamma_1^k)}{k!}(\log X)^{kr_1}+O((\log X)^{kr_1-1}).
\]
\item[(2)] Let $\AA =\{\av\}$ where $\av\in (K^*)^k$ is a tuple with \eqref{eq1.a}.
Of course, \eqref{eq1.b} need not be satisfied. But let $\AA':=\{ \sigma (\av ):\, \sigma\in S_k\}$;
then $\TT_{\{\av\}}(X)=\TT_{\AA'}(X)$, and $\AA'$ satisfies \eqref{eq1.a}, \eqref{eq1.c}, \eqref{eq1.b}. Define $G(\av )$ to be the subgroup of $\sigma\in S_k$ such that $\sigma (\av )=\av$. Then $|\AA'| =k!/|G(\av )|$. So,
\[
|\TT_{\{ \av\}}(X)|=|\TT_{\AA'}(X)|=\frac{c(\Gamma_1^k)}{|G(\av )|}(\log X)^{kr_1}+O((\log X)^{kr_1-1}).
\]
\end{itemize}

\section{Auxiliary results}\label{sec3}

In this section, we have collected the main tools for our proofs.
We start with a quantitative version of the so-called Parametric Subspace Theorem.
The first result of this type was proved in 1996 by Schlickewei \cite{schl1996}, then it was improved and generalized in 2002 by Evertse and Schlickewei \cite{evschl}. The up to now sharpest version was obtained in 2013 by Evertse and Ferretti \cite{evfer}.
The next result is a consequence of their Theorem 2.1. The following notation is used:
\begin{itemize}
\item[-]  $K$ is a number field, $S$ a finite set of places of $K$ (not necessarily containing all infinite places) and $k\geq 2$ an integer;
\item[-]
$\LL =( L_{iv}:\, v\in S,\, i=1\kdots k)$ is a  tuple of linear forms in $K[X_1\kdots X_k]$ such that for each $v\in S$, $\{ L_{iv}:\, i=1\kdots k\}$ is linearly independent;
\item[-]
$\cv =(c_{iv}:\, v\in S, i=1\kdots k)$ is a tuple of reals.
\end{itemize}
For every parameter value $Q>1$ we define the twisted height
\begin{equation}\label{eq:3.1}
H_{\LL, \cv ,Q}(\xv):=
\prod_{v\in S}\Big(\max_{1\leq i\leq k} \| L_{iv}(\xv )\|_vQ^{-c_{iv}}\Big)\cdot 
\prod_{v\in M_K\setminus S}\|\xv\|_v\ \ \text{for }\xv\in K^n.
\end{equation}

\begin{lemma}\label{lem:3.1}
There are constants $C_1,C_2>0$, depending only on $k$, $K$ and $\LL$, with the following property. Let $0<\delta <1$ and assume that
\begin{equation}\label{eq:3.2}
\sum_{i=1}^k c_{iv}=0\ \text{for } v\in S,\ \ \sum_{v\in S}\max (c_{1v}\kdots c_{kv})\leq 1.
\end{equation}
Then there are proper linear subspaces $T_1\kdots T_t$ of $K^k$, with 
\begin{equation*}
t\leq C_1\delta^{-3}(\log \delta^{-1})^2
\end{equation*}
such that for every real $Q$ with  
\begin{equation*}
Q\geq C_2^{1/\delta}
\end{equation*}
there is $T_i\in\{ T_1\kdots T_t\}$ with $\{ \xv\in K^n:\, H_{\LL ,\cv ,Q}(\xv )\leq Q^{-\delta}\}\subset T_i$.
\end{lemma}
\noindent
In fact, Theorem 2.1 of \cite{evfer} gives 
a version with fully explicit $C_1$, $C_2$. Moreover, it gives an absolute version for twisted heights defined on the algebraic closure of $\Qq$ instead of just on a number field. For our purposes, we need only the explicit dependence on $\delta$.

Condition \eqref{eq:3.2} is a normalization. For applications, the following corollary, where this normalization has been removed, is more convenient.

\begin{corollary}\label{cor:3.2}
Let $C_1,C_2$ be the constants from Lemma \ref{lem:3.1}.
Let $\lambda ,\mu, \theta$ be reals with $\lambda <\mu <\theta$. 
Further, let $\dv=(d_{iv}:\, v\in S,\, i=1\kdots k)$ be a tuple of reals such that
\begin{equation}\label{eq:3.3}
\medfrac{1}{k}\sum_{v\in S}\sum_{i=1}^k d_{iv}\leq  \lambda ,\ \ \sum_{v\in S}\max (d_{1v}\kdots d_{kv})\leq \theta .
\end{equation}
Put $\delta :=\medfrac{\mu -\lambda}{\theta -\lambda}$. 
Then there are proper linear subspaces $T_1\kdots T_t$ of $K^n$, with 
\begin{equation}\label{eq:3.4}
t\leq C_1\delta^{-3}(\log \delta^{-1})^2
\end{equation}
such that for every real $Q$ with  
\begin{equation}\label{eq:3.5}
Q\geq C_2^{1/(\mu -\lambda )}
\end{equation}
there is $T_i\in\{ T_1\kdots T_t\}$ with $\{ \xv\in K^n:\, H_{\LL ,Q,\dv}(\xv )\leq Q^{-\mu}\}\subset T_i$.
\end{corollary}

\begin{proof} We reduce to Lemma \ref{lem:3.1}.
We first observe that if we increase $\lambda$, then both the upper bound for $t$ 
in \eqref{eq:3.4}, as well as the lower bound for $Q$ in \eqref{eq:3.5} increase.
Hence we may, and shall, assume that 
\[
\medfrac{1}{k}\sum_{v\in S}\sum_{i=1}^k d_{iv}=  \lambda .
\] 
  
Clearly, $0<\delta <1$.
Define
\begin{align*}
&Q':=Q^{\theta -\lambda},
\\
&c_{iv}:=\frac{1}{\theta -\lambda}\cdot\Big( d_{iv}-\medfrac{1}{k}\sum_{i=1}^k d_{iv}\Big)
\ \ \text{for } v\in S,\, i=1\kdots k.
\end{align*}
Notice that these $c_{iv}$ satisfy \eqref{eq:3.2}. Further,  $H_{\LL ,\cv ,Q'}(\xv )=Q^{\lambda}H_{\LL ,\dv ,Q}(\xv )$.
Hence, $H_{\LL ,\dv ,Q}(\xv )\leq Q^{-\mu}$ is equivalent to 
$H_{\LL ,\cv ,Q'}(\xv )\leq Q'^{-\delta}$. 
Also, the condition $Q\geq C_2^{1/(\mu -\lambda )}$ is equivalent to $Q'\geq C_2^{1/\delta}$. 
Thus, Lemma \ref{lem:3.1} implies Corollary \ref{cor:3.2}.
\end{proof}

Henceforth, $K$ is an algebraic number field, and
$\Gamma$ a finitely generated subgroup of rank $r\geq 1$ of $(K^*)^k$, where $k\geq 2$.
Elements of $\Gamma$ are written as $\xv =(x_1\kdots x_k)$.
We fix a basis $\uv_1=(u_{11}\kdots u_{k1})\kdots \uv_r=(u_{r1}\kdots u_{rk})$ of $\Gamma/\Gamma_{{\rm tors}}$. 
Let $S$ be given by \eqref{eq:3.6}.

Corollary \ref{cor:3.2} (applied with $Q=H(\xv )$ for $\xv\in\Gamma$) is the main tool in the proof of Theorem \ref{thm:2.1}. It will be used in combination with the following covering lemma, which is Lemma 6.3.6 of \cite{egybook}.

\begin{lemma}\label{lem:3.3}
Let $\eta >0$. There is a subset $\EE$ of $\Rr^{k|S|}$ of cardinality 
\[
|\EE |\leq (1+4k/\eta )^r
\]
such that for every $\xv\in\Gamma$ there is $\ev =(e_{iv}:\, v\in S,\, i=1\kdots k)\in\EE$ with 
\begin{equation}\label{eq:3.7}
\sum_{v\in S}\sum_{i=1}^k\left| e_{iv}- \frac{\log \| x_i\|_v}{h(\xv )}\right| \leq \eta .
\end{equation}
\end{lemma}

The following lemma compares the height $H(a_1x_1+\cdots +a_kx_k)$ with $H(\xv )$, for $\xv=(x_1\kdots x_k)\in\Gamma$.

\begin{lemma}\label{lem:3.4}
Let $\Gamma$ be a finitely generated subgroup of $(K^*)^k$ of rank $r\geq 1$ and
let $\av =(a_1\kdots a_k)\in (K^*)^k$, $\varepsilon >0$. There is a constant
$C(\Gamma ,\av ,\varepsilon )>0$, depending only on $\Gamma ,\av ,\varepsilon$, such that 
\begin{equation}\label{eq:3.8}
H(\xv )\leq C(\Gamma ,\av ,\varepsilon )\cdot H(a_1x_1+\cdots +a_kx_k)^{1+\varepsilon}
\end{equation}
for all $\xv\in\Gamma$ with $\sum_{i\in I} a_ix_i\not= 0$ for each non-empty subset $I$  of $\{ 1\kdots k\}$.
\end{lemma}
\noindent
\textbf{Remark.} The constant $C(\Gamma ,\av ,\varepsilon )$ is not effectively computable from the method of proof. This lemma is used in the proofs of all theorems, 
except Theorem \ref{thm:2.2}. Consequently, 
the implied constants in the asymptotic formulas in all theorems except Theorem \ref{thm:2.2} are not effectively computable.

\begin{proof}
We apply \cite[Theorem 6.1.1]{egybook} (which is equivalent to \cite[Theorem 1]{evertse}),
which can be stated as follows. Let $T$ be a finite set of places of $K$, containing all infinite places. Let $y_0\kdots y_m$ be $T$-integers, i.e., $\| y_i\|_v\leq 1$ for $v\in M_K\setminus T$, $i=0\kdots m$ and suppose that $y_0+\cdots +y_m=0$ and that $\sum_{i\in I} y_i\not= 0$ for each proper, non-empty subset $I$ of $\{ 0\kdots m\}$.
Let $\varepsilon >0$. Then
\begin{equation}\label{eq:3.9}
\prod_{v\in T}\max (\| y_0\|_v\kdots \| y_m\|_v)\ll \Big( \prod_{v\in T}\prod_{i=0}^m\| y_i\|_v\Big)^{1+\varepsilon},
\end{equation} 
where the constant implied by the Vinogradov symbol $\ll$ depends  on $m,K,T$ and $\varepsilon$.

Now choose for $T$ the set of places obtained by adding to $S$ all infinite places of $K$ and all places $v$ of $K$ such that $\| a_i\|_v\not= 1$ for at least one $i$.
Then $k,K,T$ are determined by $\Gamma$ and $\av$.
Choose a set $\NN$ of cardinality $2^k$, consisting of positive rational integers $n$ such that $\| n\|_v=1$ for all finite places $v\in T$. We can choose $\NN$ with an upper bound 
depending only on $k$, $K$, $T$, hence only on $\Gamma$, $\av$. Let $\xv\in\Gamma$ be such that $\sum_{i\in I} a_ix_i\not= 0$ for each non-empty  $I\subseteq \{ 1\kdots k\}$. There is $n_0\in\NN$ such that $n_0+\sum_{i\in I} a_ix_i\not= 0$ for each non-empty  $I\subseteq \{ 1\kdots k\}$.
Let $m=k+1$, $y_0=n_0$, $y_i=a_ix_i$ for $i=1\kdots k$, $y_{k+1}=-n_0-\sum_{i=1}^k y_i$.
Then all conditions of the above stated \cite[Theorem 6.1.1]{egybook} are satisfied, so we can apply \eqref{eq:3.9}.
Since $\max_i\| y_i\|_v\geq 1$ for all $v\in T$, and equal to $1$ for $v\in M_K\setminus T$, the left-hand side of \eqref{eq:3.9} is $\gg H(\xv )$, while the right-hand side of \eqref{eq:3.9} is
\[
\ll \Big(\prod_{v\in T} \| n_0+a_1x_1+\cdots +a_kx_k\|_v\Big)^{1+\varepsilon}\ll H(a_1x_1+\cdots +a_kx_k)^{1+\varepsilon},
\]
where the implied constants depend on $\Gamma ,\av$, and $\varepsilon$.
This implies Lemma \ref{lem:3.4}.
\end{proof}

\begin{lemma}\label{lem:3.5}
Let $\Gamma$ be a finitely generated subgroup of $(K^*)^k$ of rank $r\geq 1$  and 
$\av =(a_1\kdots a_k)\in (K^*)^k$. Then the equation
\[
a_1x_1+\cdots +a_kx_k=1\ \ \text{in } \xv\in\Gamma
\]
has at most $c(k,r)$ solutions, with $\sum_{i\in I} a_ix_i\not= 0$ for each proper, non-empty subset of $\{ 1\kdots k\}$, where $c(k,r)$ depends on $k$ and $r$ only.
\end{lemma}
\noindent
This is the main result of \cite{evss}, with $c(k,r)$ exponential in $r$ and doubly exponential in $k$. Amoroso and Viada \cite{amvi} improved this to $c(k,r)=(8k)^{4k^4(k+r+1)}$.\qed 
\\[0.15cm]

We also need a quantitative result for linear recurrence sequences.

\begin{lemma}\label{lem:3.7}
Let $a_1\kdots a_k$ be non-zero polynomials in $K[X]$, and $\alpha_1\kdots \alpha_k\in K^*$.
Suppose that $\alpha_i/\alpha_j$ is not a root of unity for each pair $i\not= j\in\{ 1\kdots k\}$. 
Let $u_n:=a_1(n)\alpha_1^n+\cdots +a_k(n)\alpha_k^n$ for $n\in\Zz_{\geq 0}$.
Then the number of $n\in\Zz_{\geq 0}$ such that $u_n=0$ is at most $
c'(\ell )$, the latter depending only on
$\ell := k+\deg a_1 +\cdots +\deg a_k$.
\end{lemma}
\noindent
This was proved by Schmidt \cite{schmidt1999} with $c'(\ell )$
 triply exponential in $\ell$. Amoroso and Viada
\cite{amvi2011} improved this to
$c'(\ell )=\exp\exp (70\ell )$.\qed
\\[0.15cm]

We now prove some asymptotic formulas. Define the set 
\[
\HH_{\Gamma} (X):=\{ \xv\in\Gamma :\, H(\xv )\leq X\}.
\] 
Let $c(\Gamma )$ be the quantity defined by \eqref{eq:2.6}.
In the results stated below, the implied constants depend on $\Gamma$ (in fact, a priori they depend also on the choice of a basis $\{ \uv_1\kdots \uv_r\}$ of $\Gamma/\Gamma_{{\rm tors}}$, but we can choose such a basis depending on $\Gamma$, e.g., we can choose a total ordering on the set of tuples $(\uv_1\kdots \uv_k)\in (K^*)^k$, and choose the smallest basis of $\Gamma/\Gamma_{{\rm tors}}$ in this total ordering).

\begin{lemma}\label{lem:3.8}
Let $\Gamma$ be a finitely generated subgroup of $(K^*)^k$ of rank $r\geq 1$. Then
\begin{equation}\label{eq:3.10}
|\HH_{\Gamma}(X)|=c(\Gamma )(\log X)^r +O((\log X)^{r-1})\ \ \text{as } X\to\infty ,
\end{equation}
where the implied constant depends on $\Gamma$.
\end{lemma}

\begin{proof} 
We can represent elements of $\Gamma$ in the form \eqref{eq:2.3a}.
Define the map
\begin{equation}\label{eq:3.13}
\varphi :\, \Gamma\to\Zz^r :\, \mathbf{\zeta}\uv_1^{z_1}\cdots \uv_r^{z_r} \mapsto \zv =(z_1\kdots z_r).
\end{equation}
Notice that $\varphi$ is a surjective homomorphism, with kernel $\Gamma_{{\rm tors}}$.  
By \eqref{eq:2.4} we have
\begin{align*}
\varphi (\HH_{\Gamma} (X)) &=
\Big\{ \zv\in\Zz^r:\, 
\sum_{v\in S}\max \big(0,\ell_{v1}(\textbf{z})\kdots \ell_{vk}(\textbf{z})\big)\leq\log X\Big\}
\\
&= (\log X)\cdot\CC (\Gamma )\cap\Zz^r ,
\end{align*}
where $\CC (\Gamma )$ is the set defined by \eqref{eq:2.5}. 
To each $\zv\in\Zz^r$ we associate a cube of measure $1$,
\begin{equation}\label{eq:3.11}
\KK_{\zv}:=\{ \mathbf{\xi }:=(\xi_1\kdots \xi_r)\in\Rr^r:\, -\half <  \xi_i-z_i\leq \half\ \text{for } i=1\kdots r\}.
\end{equation}
Notice that these cubes cover $\Rr^r$ and are pairwise disjoint.
There is a constant $c>0$, depending only on $\Gamma$, such that
\[
(\log X -c)\CC (\Gamma )\subseteq\bigcup_{\zv\in (\log X)\CC (\Gamma )\cap \Zz^r} \KK_{\zv}\subseteq (\log X +c)\CC (\Gamma ).
\]
By comparing measures, we see that
\[
(\log X-c)^r\mu^r(\CC (\Gamma ))\leq |(\log X)\CC (\Gamma )\cap\Zz^r|
\leq (\log X+c)^r\mu^r(\CC (\Gamma )).
\]
Recalling that the map $\varphi$ defined above is $|\Gamma_{{\rm tors}}|$ to $1$,
we obtain
\[
(\log X-c)^rc(\Gamma )\leq |\HH_{\Gamma}(X)|\leq (\log X+c)^rc(\Gamma ).
\]
This implies Lemma \ref{lem:3.8}.
\end{proof}

\begin{lemma}\label{lem:3.9}
Let $\Gamma$ be a finitely generated subgroup of $(K^*)^k$ of rank $r\geq 1$, 
and assume that $\Gamma$ satisfies \eqref{eq:2.7}. 
Let $T$ be a linear subspace of $K^n$. Then
\[
|\HH_{\Gamma}(X)\cap T|\ll (\log X)^{r-1}\ \ \text{as } X\to\infty ,
\]
where the implied constant is independent of $T$ and depends only on $\Gamma$.
\end{lemma}

\begin{proof}
Assume that
$\HH_{\Gamma}(X)\cap T\not=\emptyset$.
Choose a non-zero tuple
$\bv =(b_1\kdots b_k)\in K^k$ such that $b_1x_1+\cdots +b_kx_k=0$ for $\xv\in T$.
Then at least two of the entries of $\bv$ are non-zero. 

We first show that 
there is a finite subset $\TT$ of $K^*$ (possibly depending on $\bv$) of cardinality $\leq c'(k,r)$ depending only on $k,r$ with the following property. For every $\xv\in\Gamma$ with $b_1x_1+\cdots +b_kx_k=0$ there are $i\not= j\in\{ 1\kdots k\}$ and $\lambda\in\TT$ such that $b_ib_j\not= 0$ and $x_i/x_j=\lambda$. Indeed, let 
$\xv\in\Gamma$ be such that $\sum_{i=1}^k b_ix_i =0$. Take a minimal subset $I$ of $\{ 1\kdots k\}$ such that $b_i\not= 0$ for $i\in I$ 
and $\sum_{i\in I} b_ix_i= 0$. Then $I$ has cardinality at least $2$.
Pick $j\in I$. Then 
\[
\sum_{i\in I\setminus\{ j\}} (-b_i/b_j)(x_i/x_j) =1,
\]
and no proper subsum of the left-hand side is $0$. 
The existence of $\TT$ as above follows by applying Lemma \ref{lem:3.5} to the latter identity.

Fix $i,j,\lambda$, and let $\TT_{i,j,\lambda}$ be the set of $\xv\in\HH_{\Gamma}(X)$ such that $x_i/x_j =\lambda$. 
Further, for $i\not= j$, let $\Gamma_{i,j}$ be the group of $\xv\in\Gamma$ with
$x_i=x_j$. By \eqref{eq:2.7}, this group has rank smaller than $r$.
Notice that $\lambda$ may depend on $T$.
 
Assume that $\TT_{i,j,\lambda}$ is non-empty. Pick
$\xv_{0}$ from this set, and for each $\xv\in \TT_{i,j,\lambda}$, write $\xv =\xv_0\cdot \yv$. Then $\yv\in \Gamma_{i,j}$.
Further, for $\xv\in\TT_{i,j,\lambda}$ we have,
by \eqref{eq:2.1}, \eqref{eq:2.3},
\[
H(\yv )\leq H(\xv_0^{-1}\xv)\leq X^{k+1}.
\]  
Lemma \ref{lem:3.8} implies that if $\xv$ runs through $\TT_{i,j,\lambda}$,
then $\yv$ runs through a set of cardinality 
\[
\ll (\log X)^{\rank \Gamma_{i,j}}\ll (\log X)^{r-1},
\]
where the implied constant depends only on $\Gamma$, and is independent of $\lambda$. 
 This is clearly also true if $\TT_{i,j,\lambda}$ is empty. 
Hence, $|\TT_{i,j,\lambda}|\ll (\log X)^{r-1}$. Since the number of triples $(i,j,\lambda )$ is 
at most $k^2c'(k,r)$, we have
\[
|\HH_{\Gamma}(X)\cap T|\leq  
\sum_{i,j,\lambda} |\TT_{i,j,\lambda}|\ll (\log X)^{r-1}\ \ \text{as } X\to\infty ,
\]
where the implied constant depends only on $\Gamma$ and is independent of $T$. 
\end{proof}

\begin{lemma}\label{lem:3.10}
Let $\Gamma$ be a finitely generated subgroup of $(K^*)^k$ of rank $r\geq 1$
and let $C_2>C_1>0$. Put $x_0:=1$.
\\
Let $w\in S$ and $i\not= j\in\{0\kdots k\}$ be such that there are $\xv\in\Gamma$  with  
\[
\|x_i\|_w\not= \|x_j\|_w.
\]
Then 
\begin{equation}\label{eq:3.12}
\left|\Big\{ \xv\in\HH_{\Gamma}(X) :\, C_1\leq \frac{\|x_i\|_w}{\| x_j\|_w}\leq C_2\Big\}\right|\ll (\log X)^{r-1}\ \ \text{as } X\to\infty ,
\end{equation}
where the implied constant depends on $\Gamma$, $C_1$ and $C_2$.
\end{lemma}
\noindent
\textbf{Remark.} Notice that we allow here that one of $i,j$ is $0$.

\begin{proof}
Put $c_i:=\log C_i$ for $i=1,2$.
We use the notation from \eqref{eq:2.4}, \eqref{eq:2.5} and in addition, put $\ell_{0w}(\mathbf{\xi}):=0$. Further, 
put
$m_1(\mathbf{\xi}):=\ell_{iw}(\mathbf{\xi})-\ell_{jw}(\mathbf{\xi})$.
By our assumptions on $\Gamma$, $i,j,w$, the linear form $m_1$ is not identically $0$. 
Let $\HH$ denote the set given in \eqref{eq:3.12}. The map $\varphi$ defined by \eqref{eq:3.13} maps $\HH$ to
\[
\AA\cap \Zz^r ,\ \ \text{where } \AA:=\{ \mathbf{\xi}\in (\log X)\cdot \CC (\Gamma ),\ c_1\leq m_1(\mathbf{\xi})\leq c_2\},
\]
and is $|\Gamma_{{\rm tors}}|$ to $1$. We replace $\AA$ by a larger set that is easier to handle. 
Choose linear forms $m_2\kdots m_r$ from $\ell_{vh}$ ($v\in S$, $h=1\kdots k$) such that $m_1\kdots m_r$
are linearly independent. There is a constant $C>0$, depending only on $\Gamma$, such that for $\mathbf{\xi}\in\Rr^r$ we have
\[
\sum_{v\in S}\max \big( 0,\ell_{v1}(\mathbf{\xi})\kdots \ell_{vk}(\mathbf{\xi})\big)
\geq C\cdot \max \big(|m_1(\mathbf{\xi})|\kdots  |m_r (\mathbf{\xi})|\big).
\]
This can be proved by observing first that it suffices to show this for the set of
$\mathbf{\xi}\in\Rr^r$ for which the maximum on the right-hand side is $1$, second, that this set is compact,
and third, that the left-hand side is a continuous and positive real function,  hence assumes a minimum $C>0$ on this set. Thus, it follows that
\begin{align*}
\AA \subseteq \BB:= \{ \mathbf{\xi}\in\Rr^r :\, &c_1\leq m_1(\mathbf{\xi})\leq c_2,
\\
&\qquad  
|m_i(\mathbf{\xi})|\leq C^{-1}\log X\ \text{for } i=2\kdots k\}.
\end{align*}
We estimate $|\BB\cap\Zz^r|$ from above. Consider again the cubes $\KK_{\zv}$ defined by \eqref{eq:3.11}. There is $c>0$ depending only on $\Gamma$ such that
\begin{align*}
\bigcup_{\zv\in\BB\cap \Zz^r} \KK_{\zv}\subseteq
\{ \mathbf{\xi}\in\Rr^r :\, &c_1-c\leq m_1(\mathbf{\xi})\leq c_2+c,
\\[-0.25cm]
&\qquad
|m_i(\mathbf{\xi})|\leq c+C^{-1}\log X\ \text{for } i=2\kdots k\}.
\end{align*}
It follows that
\[
|\BB\cap\Zz^r|\ll (c_2-c_1+2c)(c+C^{-1}\log X)^{r-1}\ll (\log X)^{r-1}\ \ \text{as } X\to\infty .
\]
Hence, $|\HH |\leq |\Gamma_{{\rm tors}}|\cdot |\AA\cap\Zz^r|\ll (\log X)^{r-1}$ as $X\to\infty$.
\end{proof}

\section{Proof of Theorem \ref{thm:2.1}}\label{sec4}

We keep the notation introduced in Sections \ref{sec2} and \ref{sec3}.
In particular, $K$ is a number field, $k\geq 2$, $\av =(a_1\kdots a_k)\in (K^*)^k$, 
and $\Gamma$ is a finitely generated subgroup of $(K^*)^k$ of rank $r\geq 1$.
Let $S$ be the finite set of places of $K$ given by \eqref{eq:3.6}.

Let $\varepsilon =\varepsilon (X)$ be a positive, decreasing function of $X$, to be specified later. We first estimate from above $|\VV'(X)|$, where 
\begin{align*}
\VV'(X) := &\{ \xv\in \VV_{\Gamma ,\av}(X):\, 
\\
&\quad \exists w\in S\ \text{with } 
\|a_1x_1+\cdots +a_kx_k\|_w\leq \|\xv\|_w\cdot H(\xv )^{-\varepsilon}\}.
\end{align*}
Write $\VV'(X)=\bigcup_{w\in S,\, i=1\kdots k} \VV_{i,w}(X)$, where
\begin{align*}
\VV_{i,w}(X):=\Big\{ \xv\in\VV_{\Gamma ,\av}(X):\, &\|\xv\|_w=\|x_i\|_w,
\\[-0.15cm] 
&\|a_1x_1+\cdots +a_kx_k\|_w\leq \|x_i\|_w\cdot H(\xv )^{-\varepsilon}\Big\}.
\end{align*}

Fix $i,w$. 
To estimate from above $|\VV_{i,w}(X)|$, 
we apply Corollary \ref{cor:3.2}. 
Define the system of linear forms $\LL =\{ L_{jv}:\, v\in S,\, j=1\kdots k\}$, given by
\begin{align*}
&L_{jv}:=X_j\ \ (v\in S,\, j=1\kdots k,\ (v,j)\not= (w,i)),
\\
&L_{iw}:=a_1X_1+\cdots +a_kX_k.
\end{align*}
Let $\xv\in\VV_{i,w}(X)$. For $v\in S$, $j=1\kdots k$, define $e_{jv}(\xv ):=\log\|x_j\|_v/h(\xv )$. Let $\EE$ be the set from Lemma \ref{lem:3.3}, with $\eta =\varepsilon/(2k+2)$. Then 
\begin{equation}\label{eq:4.1}
|\EE |\ll \varepsilon^{-r}, 
\end{equation}
where here and below, constants implied by $\ll$, $\gg$ and the $O$-symbols
depend on $\Gamma ,\av$ only.
Choose $\ev\in\EE$ with
\begin{equation}\label{eq:4.2}
\sum_{v\in S}\sum_{j=1}^k |e_{jv}-e_{jv}(\xv )|\leq\eta .
\end{equation}
Next, define the tuple $\dv =(d_{jv}:\, v\in S,\, j=1\kdots k)$ by 
\[
d_{jv}:=e_{jv}\ (v\in S,\, j=1\kdots k,\ (v,j)\not= (w,i)),\ \ d_{iw}:=e_{iw}-\varepsilon .
\]
Consider the set of $\xv\in\VV_{i,w}(X)$ satisfying \eqref{eq:4.2}
for some fixed $\ev\in\EE$.
Take $Q:=H(\xv )$. We first estimate from above $H_{\LL ,\dv, Q}(\xv )$. One first verifies
that for $v\in S$,
\begin{align*}
\max_{1\leq j\leq k} \| L_{jv}(\xv )\|_vQ^{-d_{jv}} 
&\leq \max_{1\leq j\leq k} \| x_j\|_vQ^{-e_{jv}}
\\
&\leq \Big(\max_{1\leq j\leq k} \|x_j\|_wQ^{-e_{jv}(\xv )}\Big)\cdot Q^{\max_{j=1}^k |e_{jv}-e_{jv}(\xv )|}
\\
&\leq Q^{\sum_{j=1}^k |e_{jv}-e_{jv}(\xv )|},
\end{align*}
while clearly, $\|\xv\|_v=1$ for $v\in M_K\setminus S$. This implies
\begin{equation}\label{eq:4.3}
H_{\LL,\dv,Q}(\xv )\leq Q^{\eta}.
\end{equation}
Next, since by the product formula, $\sum_{v\in S}\sum_{j=1}^k e_{jv}(\xv )=0$,
\begin{align}\label{eq:4.4}
\medfrac{1}{k}\sum_{v\in S}\sum_{j=1}^k d_{jv}
&=
\medfrac{1}{k}\big(\sum_{v\in S}\sum_{j=1}^k e_{jv}-\varepsilon\big)
\\
\notag
&\leq \medfrac{1}{k}\big(\sum_{v\in S}\sum_{j=1}^k |e_{jv}-e_{jv}(\xv )|-\varepsilon\big)\leq \medfrac{1}{k}(\eta -\varepsilon ).
\end{align}
Lastly,
\begin{align}\label{eq:4.5}
&\sum_{v\in S}\max (d_{1v}\kdots d_{kv}) \leq \sum_{v\in S}\max (e_{1v}\kdots e_{kv})
\\
\notag
&\leq \sum_{v\in S}\max (e_{1v}(\xv )\kdots e_{kv}(\xv ))+\sum_{v\in S}\sum_{j=1}^k |e_{jv}-e_{jv}(\xv )|\leq 1+\eta .
\end{align}
We apply Corollary \ref{cor:3.2} with $\lambda =(\eta -\varepsilon )/k$, $\mu =-\eta$, $\theta =1+\eta$, $Q=H(\xv )$. With our choice $\eta =\varepsilon/(2k+2)$, we have
\[
\mu -\lambda =\frac{\varepsilon}{2k},\ \ \
\theta -\lambda \geq 1+\varepsilon,\ \ \
\delta = \frac{\mu -\lambda}{\theta -\lambda} \geq \frac{\varepsilon}{2k(1+\varepsilon)}.
\] 
It follows that for each given $\ev\in\EE$, the set of $\xv$ such that
\[ 
\xv\in\VV_{i,w}(X),\ \text{$\xv$ satisfies \eqref{eq:4.2}},\  
H(\xv )\geq C_2^{2k/\varepsilon}
\]
is contained in a union of $\ll \varepsilon^{-3}(\log\varepsilon^{-1})^2$ proper linear subspaces of $K^n$. Taking the union over all $\ev\in\EE$, using \eqref{eq:4.1}, it follows that the set of $\xv$ with
\[
\xv\in\VV_{i,w}(X),\ \ H(\xv )\geq C_2^{2k/\varepsilon}
\]
is contained in a union of
\begin{equation}\label{eq:4.6}
\ll (\varepsilon^{-1})^{r+3}(\log\varepsilon^{-1})^2 
\end{equation}
proper linear subspaces of $K^k$.

Let $T$ be one of these subspaces.
By Lemma \ref{lem:3.4}, each $\xv\in\VV_{i,w}(X)$ has $H(\xv )\ll X^2$. Now Lemma \ref{lem:3.9} implies that 
\begin{equation}\label{eq:4.7}
|\VV_{i,w}(X) \cap T|\ll (\log X)^{r-1},
\end{equation}
where the implied constant is independent of $T$.
Lastly, by Lemma \ref{lem:3.8}, 
\begin{equation}\label{eq:4.8}
|\{\xv\in \Gamma :\, H(\xv )\leq C_2^{2k/\varepsilon}\}|\ll \varepsilon^{-r}.
\end{equation}
Combining \eqref{eq:4.6}, \eqref{eq:4.7}, \eqref{eq:4.8}, we infer
\[
|\VV_{i,w}(X)|\ll (\log X)^{r-1}
\cdot (\varepsilon^{-1})^{r+3}(\log\varepsilon^{-1})^2 . 
\]
This holds for all $i=1\kdots k$, $w\in S$. It follows that $\VV'(X)$ has cardinality
\begin{equation}\label{eq:4.9}
|\VV'(X)|\ll (\log X)^{r-1}
\cdot (\varepsilon^{-1})^{r+3}(\log\varepsilon^{-1})^2 . 
\end{equation}

We now consider $\VV_{\Gamma ,\av}(X)\setminus \VV'(X)$.
Observe that for every $\xv$ in this set,
\[
\| a_1x_1+\cdots +a_kx_k\|_v\geq \|\xv\|_vH(\xv )^{-\varepsilon}\ \ \text{for all } v\in S,
\]
while 
\[
\max (1,\|a_1x_1+\cdots +a_kx_k\|_v)\geq 1=\max (1,\| \xv\|_v)\ \ \text{for }v\in M_K\setminus S.
\] 
Hence,
\begin{equation}\label{eq:4.9c}
H(a_1x_1+\cdots +a_kx_k)\gg H(\xv )^{1-s\varepsilon}\ \ 
\text{for } \xv\in \VV_{\Gamma ,\av}(X)\setminus\VV'(X),
\end{equation}
where $s=|S|$.
On the other hand, 
\begin{equation}\label{eq:4.9d}
H(a_1x_1+\cdots +a_kx_k)\ll H(\xv )\ \ \text{for }
\xv\in\Gamma .
\end{equation}
Define
\begin{align}\label{eq:4.9a}
\HH_{\Gamma}'(X):= &\Big\{ \xv\in\Gamma :\, H(\xv)\leq X,\ 
\\
\notag
&\quad\sum_{i\in I} a_ix_i\not= 0\ \text{for each non-empty $I\subseteq\{ 1\kdots k\}$}\Big\}.
\end{align}
Then by Lemmas \ref{lem:3.8} and \ref{lem:3.9},
\begin{equation}\label{eq:4.9b}
|\HH_{\Gamma}'(X)|=c(\Gamma )(\log X)^r+O((\log X)^{r-1})\ \ \text{as } X\to\infty .
\end{equation}
Recall that solutions with vanishing subsums have been excluded from 
$\VV_{\Gamma ,\av}(X)$. From \eqref{eq:4.9c}, \eqref{eq:4.9d} it follows that
there are positive constants $C_1,C_2$, depending only on $\Gamma ,\av$, such that
\[
\HH_{\Gamma}'(C_1X)\subseteq \VV_{\Gamma ,\av}(X)
\subseteq
\HH_{\Gamma}'(C_2X^{1/(1-s\varepsilon)})\cup \VV'(X),
\]
and using \eqref{eq:4.9b}, \eqref{eq:4.9} , we obtain
\begin{align}\label{eq:4.10}
&c(\Gamma )(\log X)^r+O\big( (\log X)^{r-1}\big)\leq |\VV_{\Gamma ,\av}(X)|
\\
\notag
&\quad
\leq c(\Gamma )(\log X)^r
+O\Big(\varepsilon\cdot (\log X)^r + (\log X)^{r-1}
\cdot (\varepsilon^{-1})^{r+3}(\log\varepsilon^{-1})^2\Big). 
\end{align} 
We now choose $\varepsilon =\varepsilon (X)=\big( (\log\log X)^2/\log X\big)^{1/(r+4)}$. Then \eqref{eq:4.9} and \eqref{eq:4.10} together give the formula stated in Theorem \ref{thm:2.1},
\[
|\VV_{\Gamma ,\av}(X)|=c(\Gamma )(\log X)^r\left( 
1+O\Big(\big(\medfrac{(\log\log X)^2}{\log X}\big)^{1/(r+4)}\Big)
\right)\ \ 
\text{as } X\to\infty . 
\]
\qed

\section{Proof of Theorem \ref{thm:2.2}}\label{sec4a}

Let as before $K$ be a number field, $k\geq 2$ an integer,
$a_1\kdots a_k\in K[X]$ non-zero polynomials, and $\alpha_1\kdots \alpha_k\in K^*$, satisfying \eqref{eq:2.7a}, i.e., $\alpha_i/\alpha_j$ is not a root of unity for each pair $i\not= j\in\{ 1\kdots k\}$. Further, let 
\[
u_n:=\sum_{i=1}^k a_i(n)\alpha_i^n.
\] 
Our purpose is to find an asymptotic formula for $|\UU (X)|$, where
\[
\UU (X):=\{ n\in\Zz_{\geq 0}:\, H(u_n)\leq X\}.
\]
In this section, constants implied by $\ll$, $\gg$ and $O$-symbols will depend only on $a_1\kdots a_k$, $\alpha_1\kdots \alpha_k$, and are all effectively computable. 
Let $S$ be a finite set of places of $K$, containing all infinite places, such that the coefficients of $a_1\kdots a_k$ are all $S$-integers, while $\alpha_1\kdots \alpha_k$ are all $S$-units.
Let $\varepsilon =\varepsilon (X)$ be a positive, decreasing function of $X$, to be specified later. Put
\[
H:=H(\alpha_1\kdots \alpha_k)=\prod_{v\in M_K}\max (1,\|\alpha_1\|_v\kdots \|\alpha_k\|_v).
\]

We first estimate from above $|\UU'(X)|$, where 
\begin{align*}
\UU'(X) := &\Big\{ n\in\Zz_{\geq 0}:\, H(u_n)\leq X,
\\[-0.1cm]
&\quad \exists w\in S\ \text{with } 
\|u_n\|_w\leq \max_{1\leq i\leq k} \|a_i(n)\alpha_i^n\|_w\cdot H^{-n\varepsilon}\Big\}.
\end{align*}
Write $\UU'(X)=\bigcup_{w\in S,\, i=1\kdots k} \UU_{i,w}(X)$, where
\begin{align*}
\UU_{i,w}(X):= &\Big\{ n\in\Zz_{\geq 0}:\, H(u_n)\leq X,\, \|a_i(n)\alpha_i^n\|_w=\max_{1\leq j\leq k} \|a_j(n)\alpha_j^n\|_w,
\\[-0.15cm] 
&\quad \|u_n\|_w\leq \|a_i(n)\alpha_i^n\|_w\cdot H^{-n\varepsilon}\Big\}.
\end{align*}

Fix $i,w$. 
To estimate from above $|\UU_{i,w}(X)|$, 
we apply Corollary \ref{cor:3.2}, with
\begin{equation}\label{eq:4.11}
\xv =(x_1\kdots x_k) =(a_1(n)\alpha_1^n\kdots a_k(n)\alpha_k^n),\ \ Q=H^n.
\end{equation} 
Define the system of linear forms $\LL =\{ L_{jv}:\, v\in S,\, j=1\kdots k\}$ by
\begin{align*}
&L_{jv}:=X_j\ \ (v\in S,\, j=1\kdots k,\ (v,j)\not= (w,i)),
\\
&L_{iw}:=X_1+\cdots +X_k.
\end{align*}
Further, define
\begin{align*}
&e_{jv}:=\frac{\log \|\alpha_j\|_v}{\log H}\ (v\in S,\, j=1\kdots k),
\\
&d_{jv}:=e_{jv}\ (v\in S,\, j=1\kdots k,\, (v,j)\not= (w,i)),\ d_{iw}:=e_{iw}-\varepsilon,
\end{align*} 
and let $\dv :=(d_{jv}:\, v\in S,\, j=1\kdots k)$. 

We first estimate from above $H_{\LL ,\dv, Q}(\xv )$,
with $n\in\UU_{i,w}(X)$ and $\xv$, $Q$ as in \eqref{eq:4.11}.
For $v\in S$ we have 
\begin{align*}
\max_{1\leq j\leq k} \|L_{jv}(\xv )\|_vQ^{-d_{iv}} &\leq \max_{1\leq j\leq k}\|x_j\|_wQ^{-e_{iv}}
\leq \max_{1\leq j\leq k}\|a_j(n)\|_v,
\end{align*}
while $\|\xv\|_v\leq 1$ for $v\in M_K\setminus S$.
Then
\[
H_{\LL ,\dv,Q}(\xv ) \leq \prod_{v\in S}\max_{1\leq j\leq k}\|a_j(n)\|_v\leq Q^{\varepsilon /2k},
\]
provided that
\begin{equation}\label{eq:4.12}
n\gg \varepsilon^{-2}.
\end{equation}
Indeed, $\prod_{v\in S}\max_{1\leq j\leq k}\|a_j(n)\|_v\ll n^{O(1)}$,
and if $n$ satisfies \eqref{eq:4.12} with a sufficiently large implied constant,
this is smaller than $Q^{\varepsilon /2k}=H^{n\varepsilon /2k}$.

Next, by the product formula we have $\sum_{v\in S}\sum_{j=1}^k e_{jv}=0$, which implies
\[
\frac{1}{k}\sum_{v\in S}\sum_{j=1}^k d_{jv}=-\varepsilon /k ,
\]
while clearly,
\[
\sum_{v\in S}\max (d_{1v}\kdots d_{kv})\leq \sum_{v\in S}\max (e_{1v}\kdots e_{kv})\leq 1.
\]
We apply  Corollary \ref{cor:3.2} with $\lambda =-\varepsilon /k$, $\mu =-\varepsilon /2k$, $\theta =1$. By taking the implied constant in \eqref{eq:4.12} sufficiently large, we can guarantee that $Q=H^n$ satisfies \eqref{eq:3.5}. It follows that the set of vectors $\xv =(a_1(n)\alpha_1^n\kdots a_k(n)\alpha_k^n)$
with $n\in \UU_{i,w}(X)$ and with \eqref{eq:4.12} lies in a union of 
\[
\ll \varepsilon^{-3}(\log\varepsilon^{-1})^2
\]
proper linear subspaces of $K^n$. Consider such a subspace $T$, and take a non-trivial equation 
$b_1X_1+\cdots +b_kX_k=0$ of $T$. Then if $\xv\in T$, we have $\sum_{j=1}^k b_ja_j(n)\alpha_j^n=0$. By Lemma \ref{lem:3.7}, the number of $n\in\Zz_{\geq 0}$ satisfying such a relation is $\ll 1$. Further, the number of $n$ violating \eqref{eq:4.12} is
$\ll \varepsilon^{-2}$. It follows that altogether,
\[
|\UU_{i,w}(X)|\ll \varepsilon^{-3}(\log\varepsilon^{-1})^2.
\]
By taking the union over all $w,i$, we obtain
\begin{equation}\label{eq:4.13}
|\UU'(X)|\ll \varepsilon^{-3}(\log\varepsilon^{-1})^2.
\end{equation}

We now consider $\UU (X)\setminus \UU'(X)$.
Observe that for every integer $n$ in this set,
\[
\|u_n\|_v\geq \max_{1\leq i\leq k} \| a_i(n)\alpha_i^n\|_vH^{-n\varepsilon}\ \ \text{for } v\in S,
\]
hence
\[
\max (1,\|u_n\|_v)\geq R_v(n)\max (1,\|\alpha_1\|_v\kdots \|\alpha_k\|_v)^nH^{-n\varepsilon}\ \ \ 
\text{for } v\in S,
\]
where $R_v(n):=\min(1,\|a_1(n)\|_v\kdots \|a_k(n)\|_v)$ for $v\in S$. 
Further, the set $S$ is such that all $\alpha_i$ are $S$-units. Hence,
\[
H(u_n)\geq \big(\prod_{v\in S}R_v(n)\big) H^{n(1-s\varepsilon)}\ \ 
\text{for } n\in \UU(X)\setminus\UU'(X),
\]
where $s:=|S|$.
Let $M_v(n):=\max (1,\|a_1(n)\|_v\kdots \|a_k(n)\|_v)$ for $v\in M_K$.
Assuming $a_1(n)\cdots a_k(n)\not= 0$, we have 
by the product formula,
\begin{align*}
\prod_{v\in S}R_v(n) 
&\quad \geq \prod_{v\in S}\frac{\|a_1(n)\cdots a_k(n)\|_v}{M_v(n)^{k+1}}
\\
&\quad \geq \prod_{v\in M_K\setminus S} \|a_1(n)\cdots a_k(n)\|_v^{-1}\cdot
\prod_{v\in S}M_v(n)^{-k-1}
\\
&\quad \geq \prod_{v\in M_K}M_v(n)^{-k-1}\gg n^{O(1)}.
\end{align*}
It follows that
\[
H(u_n)\gg n^{O(1)}H^{n(1-s\varepsilon )}\ \ \text{for } n\in \UU (X)\setminus\UU'(X),
\]
where we have incorporated those $n$ with $a_1(n)\cdots a_k(n)=0$ by deminishing the implied constants.
Hence,
\begin{equation}\label{eq:4.14}
H^n\ll \big(H(u_n)(\log (2+H(u_n)))^{O(1)}\big)^{1/(1-s\varepsilon )}
\ \ \text{for } n\in \UU (X)\setminus\UU'(X).
\end{equation}
On the other hand, for $n\in \UU (X)$ we have
\begin{align}\label{eq:4.15}
H(u_n) &\ll \prod_{v\in M_K}\max (1,\|a_1(n)\alpha_1^n\|_v\kdots \|a_k(n)\alpha_k^n\|_v)
\\
\notag
&\quad \ll H(a_1(n)\kdots a_k(n))H^n\ll n^{O(1)}H^n.
\end{align}
From \eqref{eq:4.15}, \eqref{eq:4.14}, we infer that there are
effectively computable positive constants $C_1,C_2$,
$\gamma_1 ,\gamma_2$,
depending only on $a_1\kdots a_k$, $\alpha_1\kdots \alpha_k$,
such that
\begin{align*}
&\{ n\in\Zz_{\geq 0}:\, H^n\leq C_1X(\log X)^{-\gamma_1}\}\subseteq \UU (X)
\\
&\quad \subseteq
\big\{ n\in\Zz_{\geq 0}:\, H^n\leq C_2\big(X(\log X)^{\gamma_2}\big)^{1/(1-s\varepsilon )}\big\}
\cup \UU'(X).
\end{align*}
By comparing cardinalities, we obtain
\begin{align*}
&\frac{\log X}{\log H} +O(\log\log X)
\leq |\UU (X)|
\\
&\quad
\leq \frac{1}{(1-s\varepsilon )\log H}\big(\log X +O(\log\log X)\big) +O(\varepsilon^{-3}(\log \varepsilon^{-1})^2)
\\
&\quad
\leq \frac{\log X}{\log H}+O\big(\varepsilon\log X+\log\log X +\varepsilon^{-3}(\log \varepsilon^{-1})^2\big).
\end{align*}
By choosing $\varepsilon=\big( (\log\log X)^2/\log X\big)^{1/4}$, we arrive at
\[
|\UU (X)|=\frac{\log X}{\log H}\left( 1+O\Big(\Big(\frac{(\log\log X)^2}{\log X}\Big)^{1/4}\Big)\right).
\]
This completes our proof.\qed

\section{Proofs of Theorems \ref{thm:2.3a} and \ref{thm:2.3}}\label{sec5}

As before, $K$ is an algebraic number field and $k\geq 2$ an integer.

\begin{proof}[Proof of Theorem \ref{thm:2.3a}]
Let $S$ be given by \eqref{eq:3.6} and
assume that $\Gamma$ satisfies \eqref{eq:2.7b}. Here and below, the constants implied by $\ll$ and $\gg$ and the $O$-symbols depend on $\Gamma$ and $\av$.
Put $x_0:=1$ and take a constant $C_1>0$ which is chosen sufficiently large,
but depending only on $\Gamma$ and $\av$.
We first estimate from above $|\VV'(X)|$, 
where 
\[
\VV'(X):=\bigcup_{w\in S,\, i\not= j\in\{ 0\kdots k\}}\VV_{i,j,w}(X),
\]
with
\[
\VV_{i,j,w}(X):=\{ \xv\in\VV_{\Gamma ,\av}(X):\, \|x_j\|_w\leq \|x_i\|_w\leq C_1\|x_j\|_w\}
\]
(note that we allow $i,j$ to be $0$).
We estimate from above $|\VV_{i,j,w}(X)|$ for all $i,j,w$ with $i\not= j$, and subsequently, $|\VV'(X)|$.

Fix $i,j,w$. Lemma \ref{lem:3.4} implies that 
for $\xv\in\VV_{i,j,w}(X)$ we have 
\[
H(\xv )\ll H(a_1x_1+\cdots +a_kx_k)^2.
\]
Therefore, there is a constant $C_2>0$ depending only on $\Gamma$ and $\av$, such that
\[
\VV_{i,j,w}(X)\subseteq \{ \xv\in \HH_{\Gamma} (C_2X^2):
\ 
\|x_j\|_w\leq \|x_i\|_w\leq C_1\|x_j\|_w\}.
\]
By assumption \eqref{eq:2.7b},
there is $\xv\in\Gamma$ such that $\|x_i\|_w\not= \|x_j\|_w$.
So we can apply Lemma \ref{lem:3.10} to the right-hand side, and obtain
\[
|\VV_{i,j,w}(X)|\ll (\log X)^{r-1}\ \ \text{as } X\to\infty .
\]
This holds for all $i\not=j\in\{ 0\kdots k\}$ and $w\in S$. Thus, it follows that
\begin{equation}\label{eq:5.1}
|\VV'(X)|\ll (\log X)^{r-1}\ \ \text{as } X\to\infty .
\end{equation} 

We now consider $\VV_{\Gamma ,\av}(X)\setminus \VV'(X)$. Notice that for each $\xv$ in this set, and for each $v\in S$, there is $i(v)\in\{ 0\kdots k\}$ such that
\[
\max (1,\|\xv\|_v)=\|x_{i(v)}\|_v\geq C_1\max_{j\in\{ 0\kdots k\}\setminus\{ i(v)\}}\|x_j\|_v.
\]
Taking $C_1$ sufficiently large, we have for $v\in S$ that
\[
\max (1,\|a_1x_1+\cdots +a_kx_k\|_v)\gg \max (1,\|\xv\|_v)
\]
if $i(v)=0$, and
\begin{align*}
\max (1,\|a_1x_1+\cdots +a_kx_k\|_v)&\gg
 \|a_{i(v)}x_{i(v)}\|_v-\sum_{j\not=i(v)}\|a_jx_j\|_v
\\ 
& \gg \|x_{i(v)}\|_v=\max (1,\|\xv\|_v)
\end{align*}
if $i(v)\not= 0$. Further, it is clear that for $v\in M_K\setminus S$, 
\[
\max (1,\|a_1x_1+\cdots +a_kx_k\|_v)\geq 1=\max (1,\|\xv\|_v).
\]
Therefore,
\[
H(a_1x_1+\cdots +a_kx_k)\gg H(\xv )\ \ \text{for } \xv\in\VV_{\Gamma ,\av}(X)
\setminus \VV'(X).
\]
Clearly, we have also
\[
H(a_1x_1+\cdots +a_kx_k)\ll H(\xv )\ \ \text{for } \xv\in\Gamma .
\]
Thus, there are constants $C_3,C_4>0$, depending only on $\Gamma$, $\av$, such that
\[
\HH_{\Gamma }'(C_3X)\subseteq \VV_{\Gamma ,\av}(X)
\subseteq
\HH_{\Gamma}'(C_4X)\cup \VV'(X).
\]
Finally, using \eqref{eq:4.9b} and \eqref{eq:5.1},
we arrive at the formula stated in Theorem \ref{thm:2.3},
\[
|\VV_{\Gamma_1^k,\av}(X)|
=c(\Gamma )(\log X)^{r}+O((\log X)^{r-1})\ \ \text{as } X\to\infty .
\]
\end{proof}

\begin{proof}[Proof of Theorem \ref{thm:2.3}]
Let $\Gamma =\Gamma_1^k$, where $\Gamma_1$ has rank $r_1\geq 1$. So $\Gamma$ has rank $kr_1$.
Condition \eqref{eq:3.6} now becomes
\[
S=\{ v\in M_K:\, \text{there is $x\in\Gamma_1$ with $\|x\|_v\not=1$}\}.
\]
Given $v\in S$, $i\in\{ 1\kdots k\}$ we can take
$\xv =(x_1\kdots x_k)\in\Gamma_1^k$ with $\|x_i\|_v\not=1$ and $x_j=1$ for all $j\not= i$.
Thus, we see that $\Gamma =\Gamma_1^k$ satisfies \eqref{eq:2.7b}.
Now Theorem \ref{thm:2.3} follows at once from Theorem \ref{thm:2.3a} with $r=kr_1$.
\end{proof}

\section{Proof of Theorem \ref{thm:2.4}}\label{sec6}

We closely follow the proof of \cite[Theorem 2]{ftz}.
We take Theorem \ref{thm:2.3} as a starting point. As before, $K$ is a number field and $k\geq 2$
an integer.

Let $\AA$ be a finite set of tuples from $(K^*)^k$.
Let $\VV_{\AA} (X)$ be the collection of tuples $(\alpha_1\kdots \alpha_k)$ such that 
\begin{align*}
&(\alpha_1\kdots \alpha_k)=(a_1x_1\kdots a_kx_k)\ \text{with } 
(a_1\kdots a_k)\in\AA,\ x_1\kdots x_k\in \Gamma_1,
\\
&H(\alpha_1+\cdots +\alpha_k)\leq X,
\\
&\text{none of the subsums of $\alpha_1+\cdots +\alpha_k$ vanishes.}
\end{align*}
In what follows, constants implied by $O$-symbols and $\ll$-symbols will depend only on $k,\Gamma_1,\AA$.

\begin{lemma}\label{lem6.0} (i) We have for $X\to\infty$,
\[
|\VV_{\AA}(X)|\ll (\log X)^{kr_1}.
\]
(ii) If moreover $\AA$ satisfies \eqref{eq1.c}, then for $X\to\infty$, 
\[
|\VV_{\AA} (X)|= |\AA|\cdot c(\Gamma_1^k)\cdot (\log X)^{kr_1}+O((\log X)^{kr_1-1}).
\]
\end{lemma}

\begin{proof}
This follows easily from Theorem \ref{thm:2.3}. Notice that, trivially, $|\VV_{\AA}(X)|\leq\sum_{\av} |\VV_{\av}(X)|$ where the sum is taken over all $\av\in\AA$. If $\AA$ satisfies \eqref{eq1.c} then
we have equality since $\VV_{\av}(X)$, $\VV_{\bv}(X)$ are disjoint for any two distinct $\av ,\bv\in\AA$.  
\end{proof}

We assume henceforth that $\AA$ satisfies \eqref{eq1.a}, \eqref{eq1.c}, \eqref{eq1.b}.	 
Note that thanks to \eqref{eq1.b}, the set $\VV_{\AA} (X)$ is such that if $(\alpha_1\kdots \alpha_k)\in\VV_{\AA} (X)$, 
then also $(\alpha_{\sigma (1)}\kdots \alpha_{\sigma (k)})\in\VV_{\AA} (X)$
for each $\sigma\in S_k$.
Given $(\alpha_1\kdots \alpha_k)\in\VV_{\AA}(X)$, let
\[
(\alpha_1\kdots \alpha_k)_P:=\{ (\alpha_{\sigma (1)}\kdots \alpha_{\sigma (k)}):\, \sigma\in S_k\}.
\]

\begin{lemma} \label{lem6.1}
Let $\VV_{\AA}^*(X)$ be the 
set of 
$(\alpha_1\kdots \alpha_k)\in \VV_{\AA} (X)$ such that $(\alpha_1\kdots \alpha_k)_P$ has precisely $k!$ elements. Then
\[
|\VV_{\AA} (X)\setminus\VV_{\AA}^*(X)|\ll (\log X)^{(k-1)r_1}\ \ \text{as $X\to\infty$.}
\]	
\end{lemma}

\begin{proof} Let $(\alpha_1\kdots \alpha_k)=(a_1x_1\kdots a_kx_k)\in \VV_{\AA} (X)\setminus\VV_{\AA}^*(X)$. Then there are distinct $\sigma$, $\tau\in S_k$
such that $(\alpha_{\sigma (1)}\kdots \alpha_{\sigma (k)})=(\alpha_{\tau (1)}\kdots \alpha_{\tau (k)})$,
which implies that there are distinct indices $i,j$ such that $\alpha_i=\alpha_j$.
By \eqref{eq1.a} we know that $a_i=a_j$, and thus, $x_i=x_j$.
So if we replace $\alpha_i, \alpha_j$ by the single entry $2\alpha_i$, we get a tuple from $\VV_{\AA'}(X)$, where  
$\AA'$ consists of all $(k-1)$-tuples obtained by taking those tuples from $\AA$ having two equal entries, say $a_i=a_j$
for some distinct indices $i,j$, and replacing $a_i,a_j$ by the single entry $2a_i$.
Now use that by the first part of Lemma \ref{lem6.0} we have $|\VV_{\AA'}(X)|\ll (\log X)^{(k-1)r_1}$.
\end{proof}

\begin{lemma}\label{lem6.2}
Let $\VV_{\AA}^{**}(X)$ be the set of tuples $(\alpha_1\kdots\alpha_k)\in\VV_{\AA}(X)$ such that there is
$(\beta_1\kdots \beta_k)\in \VV_{\AA}(X)\setminus (\alpha_1\kdots \alpha_k)_P$ with $\sum_{i=1}^k\beta_i=\sum_{i=1}^k\alpha_i$. Then
\[
|\VV_{\AA}^{**}(X)| \ll (\log X)^{(k-1)r_1}\ \ \text{as } X\to\infty .
\]
\end{lemma}

\begin{proof}
Let $(\alpha_1\kdots\alpha_k)\in \VV_{\AA}^{**}(X)$, 
and choose $(\beta_1\kdots \beta_k)$ from the set $\VV_{\AA}(X)\setminus (\alpha_1\kdots \alpha_k)_P$ with $\sum_{i=1}^k\beta_i-\sum_{i=1}^k\alpha_i=0$. By removing vanishing subsums, we obtain that there are non-empty subsets $I$, $J$ of $\{ 1\kdots k\}$ such that 
\begin{equation}\label{eq6.0}
\sum_{i\in I}\alpha_i -\sum_{j\in J} \beta_j=0,
\end{equation}
and no proper subsum of this expression vanishes. Here, both $I$, $J$ are non-empty because of our definition of $\VV_{\AA}(X)$, and we may choose $I$ with $|I|\geq 2$, since otherwise $(\beta_1\kdots\beta_k)\in (\alpha_1\kdots\alpha_k)_P$, which we excluded.
Write $\alpha_i=a_ix_i$ for $i=1\kdots k$, where $(a_1\kdots a_k)\in\AA$ and $x_i\in \Gamma_1$
for $i=1\kdots k$.
Pick $h\in I$. Dividing \eqref{eq6.0} by $\alpha_{h}$, we get an equation
\[
1+\sum_{i\in I\setminus\{h\}}\frac{a_i}{a_{h}}\cdot \frac{x_i}{x_{h}}-\sum_{j\in J}\cdots =0,
\]
of which no proper subsum vanishes.
By e.g., Lemma \ref{lem:3.5},  there is a finite subset $\BB$ of  $\Gamma_1$, depending only on $k,\Gamma_1,\AA$, such that $x_i/x_{h}\in\BB$ for $i\in I$. So altogether, for each $(a_1x_1\kdots a_kx_k)\in\VV_{\AA}^{**}(X)$, there are $\lambda\in\BB$ and distinct indices $h,i$ in $\{ 1\kdots k\}$, 
such that $x_i/x_h=\lambda$.  

Now let
$\AA''$ be the union of the sets $\AA_{h,i,\lambda}$ ($h\not= i\in\{ 1\kdots k\}, \lambda\in\BB$), where $\AA_{h,i,\lambda}$ is the set of $(k-1)$-tuples obtained by taking the tuples $(a_1\kdots a_k)$ from $\AA$,
and for each of these tuples,  replacing the entries $a_h,a_i$ by one single entry $a_h-\lambda a_i$. 
Then by the first part of Lemma \ref{lem6.0}, we obtain
\[
|\VV_{\AA}^{**}(X)|\leq |\VV_{\AA''}(X)|\ll (\log X)^{(k-1)r_1}.
\] 
\end{proof}

\begin{proof}[Proof of Theorem \ref{thm:2.4}]
Let $\TT^*_{\AA}(X)$ be the set of $\alpha$ satisfying \eqref{eq1.2} but with the additional property
that no subsum of $\sum_{i=1}^k a_ix_i$ vanishes. This is the set of $\alpha$ such that there is 
$(\alpha_1\kdots\alpha_k)\in\VV_{\AA}(X)$ with $\alpha =\alpha_1+\cdots +\alpha_k$. Lemmas
\ref{lem6.1} and \ref{lem6.2} imply that 
with $\ll (\log X)^{(k-1)r_1}$ exceptions, 
every $\alpha\in\TT^*_{\AA} (X)$
has precisely $k!$ different representations in the form $\alpha_1+\cdots +\alpha_k$ 
with $(\alpha_1\kdots\alpha_k)\in\VV_{\AA} (X)$, all lying in a set $(\alpha_1\kdots \alpha_k)_P$.
Thus, using the second part of Lemma \ref{lem6.0},
\begin{align}\label{eq6.x}
|\TT_{\AA}^* (X)| &=(k!)^{-1}|\VV_{\AA} (X)|+O((\log X)^{(k-1)r_1})
\\
\notag
&=\frac{|\AA|}{k!}c(\Gamma_1^k)(\log X)^{kr_1}+O((\log X)^{kr_1-1})
\ \ \text{as } X\to\infty .
\end{align}

It remains to dispose of the non-vanishing subsum condition.
Let $\TT_{\AA}^{**}(X)$ be the set of $\alpha$ satisfying \eqref{eq1.2} but now with the additional condition that every representation $\alpha =\sum_{i=1}^k a_ix_i$ with $(a_1\kdots a_k)\in\AA$,
$x_i\in \Gamma_1$ for $i=1\kdots k$ has a vanishing subsum. Pick $\alpha\in \TT_{\AA}^{**}$ with $\alpha\not= 0$.
Then $\alpha$ is in fact representable as $\sum_{i\in I} a_ix_i$, where $2\leq |I|<k$ and where no subsum of $\sum_{i\in I} a_ix_i$ vanishes.
Letting $\AA_I$ consist of all tuples $(a_i:\, i\in I)$ with $(a_1\kdots a_k)\in\AA$, and including the case $\alpha =0$, we deduce from part 1 of Lemma \ref{lem6.0} that
\[
|\TT_{\AA}^{**}(X)|\leq 1+\sum_{I\subsetneq\{ 1\kdots r\}} |\VV_{\AA_I}(X)|\ \ 
\ll (\log X)^{(k-1)r_1}\ \ \text{as } X\to\infty .
\]
By combining this with \eqref{eq6.x}, the proof of Theorem \ref{thm:2.4} is completed.
\end{proof}

\end{document}